\documentclass[a4paper,reqno]{amsart}
\usepackage{geometry}
\usepackage{mathrsfs}
\usepackage{syntonly}
\usepackage{amsmath}
\usepackage{amsthm}
\usepackage{amsfonts}
\usepackage{amssymb}
\usepackage{latexsym}
\usepackage{amscd,amssymb,amsopn,amsmath,amsthm,graphics,amsfonts,mathrsfs,accents,enumerate,verbatim,calc}
\usepackage[dvips]{graphicx}
\usepackage[colorlinks=true,linkcolor=blue,citecolor=blue]{hyperref}
\input xy
\xyoption{all}

\usepackage{setspace}

\usepackage{color}

\numberwithin{equation}{section}

\theoremstyle{plain}

\newtheorem{thm}{Theorem}[section]
\newtheorem{cor}[thm]{Corollary}
\newtheorem{lem}[thm]{Lemma}
\newtheorem{prop}[thm]{Proposition}

\newtheorem{defn}[thm]{Definition}
\newtheorem{exm}[thm]{Example}

\newtheorem{rem}[thm]{Remark}

\newdir{ >}{{}*!/-5pt/\dir{>}}

\newcommand{\Hom}{\operatorname{Hom}\nolimits}

\newcommand{\End}{\operatorname{End}\nolimits}

\newcommand{\op}{\operatorname{op}\nolimits}

\renewcommand{\mod}{\mathsf{mod}\hspace{.01in}}

\newcommand{\B}{\mathcal B}
\newcommand{\uB}{\underline{\B}}

\newcommand{\U}{\mathcal U}
\newcommand{\V}{\mathcal V}

\newcommand{\W}{\mathcal W}
\newcommand{\h}{\mathcal H}

\newcommand{\s}{\mathcal S}
\newcommand{\T}{\mathcal T}

\newcommand{\D}{\mathcal D}

\newcommand{\X}{\mathscr X}
\newcommand{\Y}{\mathscr Y}

\newcommand{\C}{\mathcal C}

\newcommand{\EE}{\mathbb E}

\newcommand{\svecv}[2]{\left(\begin{smallmatrix}
      #1 \\
      #2
    \end{smallmatrix}\right)}

\newcommand{\svech}[2]{\left(\begin{smallmatrix}
      #1 & #2
\end{smallmatrix}\right)}

\renewcommand{\emph}{\textit}
\renewcommand{\phi}{\varphi}

\newcommand{\add}{\mathsf{add}\hspace{.01in}}
\begin{document}

\title{Hearts of twin cotorsion pairs revisited: Integral and Abelian Hearts}\footnote{\hspace{1em}Yu Liu is supported by the National Natural Science Foundation of China (Grant No. 12171397). Panyue Zhou is supported by the Hunan Provincial Natural Science Foundation of China (Grant No. 2023JJ30008) and by the National Natural Science Foundation of China (Grant No. 11901190). }
\author{Yu Liu and Panyue Zhou}
\address{ School of Mathematics and Statistics, Shaanxi Normal University, 710062 Xi'an, Shannxi, P. R. China}
\email{recursive08@hotmail.com}

\address{School of Mathematics and Statistics, Changsha University of Science and Technology, 410114 Changsha, Hunan, P. R. China}
\email{panyuezhou@163.com}

\begin{abstract}
Hearts of twin cotorsion pairs are shown to be quasi-abelian. But they are not always integral. In this article,
we provide a sufficient and necessary condition for the hearts of twin cotorsion pairs being integral (resp. abelian).

\end{abstract}
\maketitle

\section{Introduction}

The notion of a cotorsion pair was first introduced by Salce \cite{Sa}. Since then, it has been extensively studied in representation theory, particularly in tilting theory and Cohen-Macaulay modules. On triangulated categories, cotorsion pair unifies the notion of $t$-structures in the sense of \cite{BBD}, co-$t$-structures in the sense of Pauksztello \cite{P}, and the notion of cluster tilting subcategories in the sense of Keller-Reiten \cite{KR} (see also \cite{BMRRT}).

Let $\C$ be a triangulated category with shift functor $[1]$. A \emph{cotorsion pair} on $\C$ is a pair of full subcategories $(\U,\V)$ such that $\Hom_\C(\U,\V[1])=0$ and any object $C\in \C$ admits a triangle $V\to U\to C\to V[1]$ with $U\in \U$, $V\in \V$. In fact, $(\U,\V)$ is a cotorsion pair if and only if $(\U,\V[1])$ is a torsion pair. Nakaoka \cite{N1} defined the hearts of cotorsion pairs on triangulated categories and showed that they are abelian.

Assume that $\C$ is Krull-Schmidt, Hom-finite. If $T$ is a cluster tilting object in $\C$, then $(\add T[1],\add T[1])$ is a cotorsion pair. By the result of Koenig and Zhu \cite[Theorem 3.3]{KZ}, the ideal quotient $\C/\add(T[1])$ is abelian. In fact, it is the heart of $(\add(T[1]),\add(T[1]))$ and is equivalent to $\mod \End_\C(T)^{\op}$. However, when $T$ is just rigid, generally speaking, $\C/\add(T[1])$ is no longer abelian. When $\C$ has a Serre functor, Buan and Marsh \cite{BM} considered the quotient $\C/\X_T$ instead, where $\X_T=\{C\in \C ~|~ \Hom_\C(T,-)=0\}$. Note that if $T$ is cluster tilting, we have $\X_T=\add(T[1])$. Unfortunately, even in this case,
$\C/\X_T$ is not abelian in general. But they showed that $\C/\X_T$ still admits a good structure. This quotient category is integral, which means that the localization of this category with respect to the epic-monic morphisms is abelian (it is equivalent to $\mod \End_\C(T)^{\op}$).

Integral category is a subclass of preabelian category which was first studied systematically by Rump \cite{R}. One of Rump's key observations was that the class of epic-monic morphisms in an integral category can be treated by using  calculus of fractions (defined by Gabriel and Zisman \cite{GM}). As a result, the canonical localization of the category at this class is an abelian category. Now we recall the definition of integral category (for the convenience of the readers, we also recall the definitions of semi-abelian and quasi-abelian category).

\begin{defn}
An additive category is called preabelian if any morphism admits a kernel and a cokernel. Let $\mathcal A$ be a preabelian category and the following diagram
$$\xymatrix{
A \ar[r]^{\mathbf{a}} \ar[d]_{\mathbf{b}} &B \ar[d]^{\mathbf{c}}\\
C \ar[r]_{\mathbf{d}} &D}
$$
is a pull-back diagram.
\begin{itemize}
\item[(1)] $\mathcal A$ is called left semi-abelian if $\mathbf{a}$ is an epimorphism whenever $\mathbf{d}$ is a cokernel;
\item[(2)] $\mathcal A$ is called left integral if $\mathbf{a}$ is an epimorphism whenever $\mathbf{d}$ is an epimorphism;
\item[(3)] $\mathcal A$ is called left quasi-abelian if $\mathbf{a}$ is a cokernel whenever $\mathbf{d}$ is a cokernel;
\end{itemize}
Dually, we can define right integral (resp. right semi-abelian, right quasi-abelian) category. $\mathcal A$ is called integral (resp. semi-abelian, quasi-abelian) if it is both left and right integral (resp. semi-abelian, quasi-abelian).
\end{defn}

We have the following relation:
$$\{ \mbox{Abelian Cat.} \}\subset\{ \mbox{Integral Cat.} \}\subset\{ \mbox{Semi-Abelian Cat.} \}\subset\{ \mbox{Preabelian Cat.} \}.$$
Under the condition Buan and Marsh considered, we can find a pair of cotorsion pairs $((\add T[1],\X_T)$, $(\X_T,\V))$. It is a special case of the concept \emph{twin cotorsion pair} defined in \cite{N2}. The ideal quotient $\C/\X_T$ is the heart of the twin cotorsion pair (also defined in \cite{N2}) $((\add T[1],\X_T),(\X_T,\Y))$. A pair of cotorsion pairs $((\s,\T),(\U,\V))$ is called a twin cotorsion pair if $\s\subseteq \U$. A cotorsion pair $(\U,\V)$ can be realized as a twin cotorsion pair $((\U,\V),(\U,\V))$. Heart of twin cotorsion pair is a generalization of heart of cotorsion pair, but it is not always abelian (it is even not integral in general). A result showed in \cite{N2} implies that if we have a twin cotorsion pair $((\s,\T),(\T,\V))$, then its heart is integral (the twin cotorsion pair hidden in \cite{BM} admits this condition). It is reasonable to consider further when the heart of a twin cotorsion pair becomes integral, since we can get an abelian category through localization. In this article, we will give an `` if and only if " condition for the hearts being integral (resp. abelian), under a more general setting.


Nakaoka and Palu \cite{NP} introduced the concept of extriangulated categories, which combines the characteristics of both triangulated and exact categories. On extriangulated categories, various concepts can be defined, such as (twin) cotorsion pairs \cite{NP}, their hearts \cite{LN}, all of which have similarities to their definitions on triangulated and exact categories.
Now we recall the definition of the heart of a twin cotorsion pair on an extriangulated category
from \cite{LN}. Let $(\B,\EE,\mathfrak{s})$ be an extriangulated category.

\begin{defn}\label{He}
For a twin cotorsion pair $((\s,\T),(\U,\V))$ on $\B$, let
\begin{itemize}
\item[(i)] $\B^+=\{X\in \B \text{ }|\text{ } \mbox{there exists an}~\text{ } \EE\text{-triangle } V\to W\to X\dashrightarrow \text{, }V\in \V \text{ and }W\in \U\cap\T \}$,
\item[(ii)] $\B^-=\{Y\in \B \text{ }|\text{ } \mbox{there exists an}~ \text{ } \EE\text{-triangle } Y\to W'\to S\dashrightarrow \text{, }S\in \s \text{ and }W'\in \U\cap\T \}$,
\item[(iii)] $\h=\B^+\cap \B^-$.
\end{itemize}
We call the ideal quotient $\h/(\U\cap \T)$ the heart of $((\s,\T),(\U,\V))$.  For a single cotorsion pair $(\U',\V')$, the heart of twin cotorsion pair $((\U',\V'),(\U',\V'))$ is called the heart of the $(\U',\V')$.
\end{defn}

For convenience, denote $\h/(\U\cap \T)$ by $\underline \h$. Hassoun and Shah \cite{HS} provided
a sufficient condition for the heart being integral, which unifies the results on triangulated categories \cite[Theorem 6.3]{N2} and exact categories \cite[Theorem 6.2]{L1}.

\begin{thm}\label{1}\cite[Theorem 3.7]{HS}
For a twin cotorsion pair $((\s,\T),(\U,\V))$ on $\B$,
\vspace{1mm}

{\rm (1)} if $\B$ has enough projectives $\mathcal P$, $\mathcal P\subseteq \U\cap\T$ and $\U\subseteq\s\ast\T$, then $\underline{\h}$ is integral;
\vspace{1mm}

{\rm (2)} if $\B$ has enough  injectives $\mathcal I$, $\mathcal I\subseteq \U\cap\T$ and $\T\subseteq\U\ast\V$, then $\underline{\h}$ is integral.
\end{thm}

In this article, we first provide an `` if and only if " condition for the hearts being integral.


\begin{thm}{\rm (see Definition \ref{epi} and Theorem \ref{main1} for details)}
For a twin cotorsion pair $((\s,\T),(\U,\V))$ on $\B$, $\underline \h$ is integral if and only if $\B^-\cap(\T*{_{\rm epi.}}\U)\subseteq \U$.
\end{thm}
%
Some `` if and only if " conditions for the hearts being abelian are discussed for the twin cotorsion pairs induced by $n$-cluster ($n>2$) tilting subcategories (see \cite{L2} and \cite{HZ}); for general case in \cite{LYZ}. Our second main result provides another sufficient and necessary condition for the heart of a twin cotorsion pair being abelian.

\begin{thm}{\rm (see Definition \ref{epi} and Theorem \ref{main2} for details)}
For a twin cotorsion pair $((\s,\T),(\U,\V))$ on $\B$, denote by $\underline {\h_1}$ the heart of $(\s,\T)$ and
$\underline {\h_2}$ the heart of $(\U,\V)$. Then
$\underline \h$ is abelian if and only if the following conditions are satisfied:
$${\rm (1)}~\underline \h=\underline {\h_1}\cap \underline {\h_2};\quad {\rm (2)}~ {_{\rm epi.}}\U\subseteq\s\oplus\W;\quad {\rm (3)}~ \T_{\rm mono.}\subseteq\V\oplus\W.$$

\end{thm}

This paper is organized as follows. In Section 2, we review some elementary concepts and properties of cotorsion pairs on extriangulated categories.
In Section 3, we provide a sufficient and necessary condition for the heart of a twin cotorsion pair
 being integral. In Section 4, we provide a  sufficient and necessary condition for the heart of a twin cotorsion pair being abelian.

\section{Preliminaries}

In this article, let $(\B,\EE,\mathfrak{s})$ be an extriangulated category (we omit the definition and basic properties of extriangulated categories, see \cite[Sections 2, 3]{NP} if needed). When we say that $\C$ is a subcategory of $\B$, we always assume that $\C$ is full and closed under isomorphisms. We first recall the definition of a (twin) cotorsion pair.

\begin{defn}\cite[Definitions 2.1 and 4.12]{NP}
Let $\U$ and $\V$ be two subcategories of $\B$ which are closed under direct summands. We call $(\U,\V)$ a \emph{cotorsion pair} if the following conditions are satisfied:
\begin{itemize}
\item[(a)] $\EE(\U,\V)=0$.

\item[(b)] For any object $B\in \B$, there exist two $\EE$-triangles
\begin{align*}
V_B\rightarrow U_B\rightarrow B{\dashrightarrow},\quad
B\rightarrow V^B\rightarrow U^B{\dashrightarrow}
\end{align*}
satisfying $U_B,U^B\in \U$ and $V_B,V^B\in \V$.
\end{itemize}
A pair of cotorsion pairs $((\s,\T),(\U,\V))$ is called a twin cotorsion pair if $\s\subseteq \U$.
\end{defn}

From now on, let $((\s,\T),(\U,\V))$ be a twin cotorsion pair and $\W=\U\cap \T$. The following remark is helpful.

\begin{rem}
For a twin cotorsion pair $((\s,\T),(\U,\V))$, we have:
\begin{itemize}
\item[(1)] $\V\subseteq \T$ and $\EE(\s,\V)=0$;
\item[(2)] $\s,\T,\U,\V$ are closed under extensions;
\end{itemize}
\end{rem}

We denote by $\X(A,B)$ the subgroup of $\Hom_\B(A,B)$ from $A$ to $B$ which factor through objects lie in a subcategory $\X$. The family of
such subgroups forms an ideal of $\B$. Then we
have a category $\B/\X$ whose objects are the objects of
$\B$ and the set of morphisms from $A$ to $B$ is $\Hom_\B(A,B)/\X(A,B)$. For convenience, for any subcategory $\D$ of $\B$, we denote its image in $\B/\W$ by the canonical quotient functor $\pi:\B\to \B/\W$ by $\underline \D$, denote the image of any morphism $f$ in $\uB$ by $\underline f$.
\vspace{1mm}

\begin{lem}\label{ft}
Let $A\in \h$. It admits two $\EE$-triangles
$$V_A\to W_A\xrightarrow{w} A\dashrightarrow,\quad A\xrightarrow{w'} W^A\to S^A\dashrightarrow$$
where $S^A\in\s$, $V_A\in\V$ and $W_A,W^A\in \W$.
Then
\begin{itemize}
\item[(1)] any morphism from an object $U\in \U$ to $A$ factors through $w$;
\item[(2)] any morphism from $A$ to an object $T\in \T$ factors through $w'$.
\end{itemize}
\end{lem}

\begin{proof}
We only show (1), (2) is by dual.

Let $a:U\to A$ be any morphism. Since $\EE(U,V_A)=0$,  from the following exact sequence
$$\Hom_{\B}(U,W_A)\xrightarrow{\Hom_{\B}(U,w)} \Hom_{\B}(U,A)\longrightarrow \EE(U,V_A)=0$$
we can obtain that there exists a morphism $u:U\to W_A$ such that $wu=a$.
\end{proof}

\begin{lem}\cite[Lemmas 2.9, 2.10 and 2.30]{LN}\label{+-}
\begin{itemize}
\item[(1)] In the $\EE$-triangle $T\to A\xrightarrow{f} B\dashrightarrow$ with $T\in \T$, $A\in \B^+$ whenever $B\in \B^+$;
\item[(2)] In the $\EE$-triangle $T\to A\xrightarrow{f} B\dashrightarrow$ with $T\in \T$ and $A,B\in \h$, $\underline f$ is a monomorphism in $\underline \h$;
\item[(3)] In the $\EE$-triangle $X\xrightarrow{x} Y\to U\dashrightarrow$ with $U\in \U$, $Y\in \B^-$ whenever $X\in \B^-$;
\item[(4)] In the $\EE$-triangle $X\xrightarrow{x} Y\to U\dashrightarrow$ with $U\in \U$  and $X,Y\in \h$, $\underline x$ is an epimorphism in $\underline \h$.
\end{itemize}
\end{lem}

\begin{lem}\cite[Proposition 1.20]{LN}\label{L4}
Let $A\xrightarrow{x} B\xrightarrow{y} C\dashrightarrow$ be any $\EE$-triangle. Let $f:A\to D$ be any morphism. Then we can get a commutative diagram of $\EE$-triangles
$$\xymatrix{
A\ar[r]^x \ar[d]_{f} &B\ar[r]^y \ar[d]^{g} &C\ar@{-->}[r] \ar@{=}[d]&\\
D \ar[r]_d &E\ar[r]_e &C\ar@{-->}[r] &
}
$$
and an $\EE$-triangle $A \xrightarrow{\svecv{f}{x}} D\oplus B \xrightarrow{\svech{-d}{g}} E \dashrightarrow$. On the other hand, if we already have a commutative diagram of $\EE$-triangles
$$\xymatrix{
A\ar[r]^x \ar[d]_{f} &B\ar[r]^y \ar[d]^{g} &C\ar@{-->}[r] \ar@{=}[d]&\\
D \ar[r]_d &E\ar[r]_e &C\ar@{-->}[r] &
}
$$
then there exists a morphism $g':B\to E$ such that when we replace $g$ by $g'$, the diagram is still commutative and $A \xrightarrow{\svecv{f}{x}} D\oplus B \xrightarrow{\svech{-d}{g'}} E \dashrightarrow$ is an $\EE$-triangle.
\end{lem}

In this section, we show several lemmas related to the hearts of twin cotorsion pairs.

\begin{lem}\label{int}
$\h\cap \U=\W=\h\cap \T$.
\end{lem}

\proof By Definition \ref{He}, we have $\W\subseteq\B^{+}$ and $\W\subseteq\B^{-}$,
 which implies that $\W\subseteq\B^{+}\cap\B^{-}=\h$. Hence $\W\subseteq \h\cap \U$.
\vspace{0.5mm}

For any object $X\in\h\cap\U$, there exists an $\EE$-triangle
$V\to W\to X\dashrightarrow $ where $V\in \V \text{ and }W\in \W$.
Since $\EE(\U,\V)=0$, we have $\EE(X,V)=0$.
Hence this $\EE$-triangle splits and $W\cong X\oplus V$. Since $\W$ is closed under direct summands,
we have $X\in\W$. Thus we obtain $\h\cap\U\subseteq \W$.
Dually, we can show that $\W=\h\cap \T$.
\qed
\vspace{2mm}

\begin{defn}
Let $f:A\to B$ be any morphism in $\B$.
\begin{itemize}
\item[(a)] $f$ is called $\W$-monic if $\Hom_\B(B,W)\xrightarrow{\Hom_\B(f,W)}\Hom_\B(A,W)$ is surjective for any $W\in \W$;
\item[(b)] $f$ is called $\W$-epic if $\Hom_\B(W,A)\xrightarrow{\Hom_\B(W,f)}\Hom_\B(W,B)$ is surjective for any $W\in \W$;
\end{itemize}
\end{defn}

\begin{rem}
In Lemma \ref{ft}, $w$ is $\W$-monic and $w'$ is $\W$-epic.
\end{rem}

\begin{lem}\label{L2}
Let $A\xrightarrow{x} B\xrightarrow{y} C\dashrightarrow$ be an $\EE$-triangle such that $A,C\in\B^-$ and $x$ is $\W$-monic. Then $B\in \B^-$.
\end{lem}

\begin{proof}
Since $A\in \B^-$, it admits an $\EE$-triangle $A \xrightarrow{w} W\to S \dashrightarrow$ with $W\in \W$ and $S\in \s$. Then $w$ factors through $x$ since $x$ is $\W$-monic. Hence we have the following commutative diagram
$$\xymatrix{
A \ar[r]^x \ar[d]_w &B\ar[r]^y \ar[d] &C \ar@{=}[d] \ar@{-->}[r] &\\
W \ar[r] \ar[d] &W\oplus C \ar[r] \ar[d] &C \ar@{-->}[r] &\\
S \ar@{-->}[d] \ar@{=}[r] &S \ar@{-->}[d]\\
&&
}
$$
of $\EE$-triangles.
By Lemma \ref{+-}, the second column of the above diagram implies that $B\in \B^-$.
\end{proof}

Let $f\colon A\to B$ be a morphism in $\h$. Since $A$ admits an $\EE$-triangle $A \xrightarrow{w} W\to S \dashrightarrow$ with $W\in \W$ and $S\in \s$, we have the following commutative diagram
$$\xymatrix{
A \ar[r]^w \ar[d]_f &W \ar[r] \ar[d] &S \ar@{=}[d] \ar@{-->}[r] &\\
B \ar[r] &C_f \ar[r] &S \ar@{-->}[r] &
}
$$
which induces an $\EE$-triangle $A\xrightarrow{\svecv{f}{w}} B\oplus W\to C_f\dashrightarrow$. Since $w$ is $\W$-monic, $\svecv{f}{w}$ is also $\W$-monic.

By \cite[Corollary 2.26]{LN}, we know that $\underline f$ is an epimorphism in $\underline \h$ if and only if $C_f\in \U$. We give the following definition.

\begin{defn}\label{epi}
Denote by ${_{\rm epi.}}\U$ the subcategory consisting of objects $U\in \U$ admitting an $\EE$-triangle
$$A\xrightarrow{~f~} B\longrightarrow U\dashrightarrow$$
which satisfies the following conditions:
\begin{itemize}
\item[(1)] $A,B\in \h$;
\item[(2)] $f$ is $\W$-monic.
\end{itemize}
Such $\EE$-triangle is called an epi-$\EE$-triangle (since $\underline f$ is an epimorphism in $\underline \h$ by {\rm Lemma \ref{+-}}). Dually we can define $\T_{\rm mono.}$ and mono-$\EE$-triangles.
\end{defn}

%
%

The following lemma is closely related to Definition \ref{epi}.

\begin{lem}\label{L3}
Let $A\xrightarrow{~f~} B\longrightarrow C\dashrightarrow$ be an $\EE$-triangle such that
\begin{itemize}
\item[(1)] $A,B\in \h$;
\item[(2)] $f$ is $\W$-monic;
\item[(3)] $\underline f$ is an epimorphism in $\underline \h$.
\end{itemize}
Then $C\in \U$, which means this $\EE$-triangle is an epi-$\EE$-triangle.
\end{lem}

\begin{proof}
$f\colon A\to B$ admits an $\EE$-triangle
$$A\xrightarrow{\svecv{f}{w}} B\oplus W\longrightarrow C_f\dashrightarrow$$
where $W\in\W$ and $C_f\in \U$. Since $f$ is $\W$-monic, there exists
a morphism $w'\colon B\to W$ such that $w'f=w$.
Thus we have the following commutative diagram.
$$\xymatrix{
A \ar[r]^f \ar@{=}[d] &B \ar[r] \ar[d]^{\svecv{1}{w'}} &C \ar[d]^g \ar@{-->}[r] &\\
A\ar[r]^{\svecv{f}{w}\quad} \ar@{=}[d] &B\oplus W \ar[r] \ar[d]^{(1~0)} &C_f \ar[d]^h \ar@{-->}[r] &\\
A \ar[r]^f&B \ar[r] &C\ar@{-->}[r] &
}
$$
By \cite[Corollary 3.6]{NP},
we have that $hg$ is an isomorphism.  There exists an automorphism $h'\in \End_{\B}(C)$ such that $h'hg=1$. Since $C$ admits an $\EE$-triangle
$$V_C\to U_C\to C\dashrightarrow\hspace{1cm}(\maltese)$$
with $V_C\in \V$ and $U_C\in \U$, by Lemma \ref{ft}, we have the following commutative diagram
$$\xymatrix{
&C \ar[r]^{g} &C_f \ar[d]^{h'h} \ar@{.>}[dl]\\
V_C \ar[r] &U_C \ar[r] &C \ar@{-->}[r] &
}
$$
which implies that $(\maltese)$ splits. Hence $U_C\cong V_C\oplus C$ and $C\in \U$.
\end{proof}

\subsection{kernels in $\underline \h$}

Let $f:A\to B$ be any morphism in $\h$. Since $B$ admits an $\EE$-triangle $V_B\to W_B\xrightarrow{w} B\dashrightarrow$ with $V_B\in \V$ and $W_B\in \W$, we can obtain the following commutative diagram
$$\xymatrix{
V_B \ar[r]^v \ar@{=}[d] &C \ar[r]^g \ar[d]^{w'} &A \ar[d]^f \ar@{-->}[r] &\\
V_B \ar[r] &W_B \ar[r]_w &B \ar@{-->}[r] &
}
$$
which induces an $\EE$-triangle $C \xrightarrow{\svecv{g}{-w'}} A\oplus W_B \xrightarrow{\svech{f}{w}} B\dashrightarrow$. By Lemma \ref{+-}, $C\in \B^+$. We have the following lemma.

\begin{lem}\label{unique}
Let $d:D\to A$ be any morphism in $\h$. If $\underline {fd}=0$,  then there exists a morphism $c:D\to C$ such that $d=gc$. Moreover, $\underline c$ is unique in $\underline \h$.
\end{lem}

\begin{proof}
Since $\underline {fd}=0$, we have the following commutative diagram
$$\xymatrix{
D \ar[r]^{w_1} \ar[d]_-{\svecv{d}{0}} &W \ar[d]^{w_2}\\
A\oplus W_B \ar[r]_-{\svech{f}{w}} &B
}
$$
with $W\in \W$. There exists a morphism $w_3:W\to W_B$ such that $ww_3=w_2$. Hence $\svech{f}{w} \circ \svecv{d}{-w_3w_1}=0$ and there exists a morphism $c:D\to C$ such that $\svecv{d}{-w_3w_1}=\svecv{g}{-w'}\circ c$.  Then $d=gc$. If there exists another morphism $c':D\to C$ such that $\underline d=\underline {gc'}$, we have that $g(c-c')$ factors through an object $W'\in \W$. Then we have the following commutative diagram
$$\xymatrix{
D \ar[r]^{w_1'} \ar[d]_-{c-c'} &W' \ar[d]^{w_2'}\\
C \ar[r]_g &A
}
$$
Since $\EE(W',V_B)=0$, there is a morphism $w_3':W'\to C$ such that $w_2'=gw_3'$. Hence $g((c-c')-w_3'w_1')=0$ and there is a morphism $v':D\to V_B$ such that $vv'=(c-c')-w_3'w_1'$. By Lemma \ref{ft}, $v'$ factors through $\W$, hence $\underline c=\underline {c'}$.
\end{proof}

Now we can find the kernel of $\underline f$ in $\underline{\h}$. Note that $C$ admits a commutative diagram
$$\xymatrix{
V_1 \ar@{=}[r] \ar[d] &V_1 \ar[d]\\
C^- \ar[r] \ar[d]_{c^-} &W_1 \ar[r] \ar[d] &S_1 \ar@{-->}[r] \ar@{=}[d] &\\
C \ar[r] \ar@{-->}[d] &T_1 \ar[r] \ar@{-->}[d] &S_1 \ar@{-->}[r] &\\
&&
}
$$
with $T_1\in \T$, $S_1\in \s$, $V_1\in \V$, $W_1\in \W$ and $C^-\in \h$. By \cite[Proposition 2.19]{LN} and the dual of \cite[Proposition 2.15]{LN}, for any morphism $h:D_1\to C$ in $\h$, there exists a unique morphism $\underline {c_1}:D_1\to C^-$ such that $\underline {c^-c_1}=\underline h$. Now by Lemma \ref{unique}, $\underline {gc^-}$ is the kernel of $\underline f$ in $\underline \h$.

\section{Integral Hearts}

In this section, we provide a sufficient and necessary condition for the hearts being integral.

Let $\X$ and $\Y$ be two full subcategories of $\B$.
We denote by $\X\ast\Y$ the subcategory consisted by the objects $Z$
such that there exists an $\EE$-triangle
$X\to Z \to Y\dashrightarrow $ with $X\in\X$ and $Y\in\Y$.

\begin{thm}\label{main1}
$\underline \h$ is integral if and only if $\B^-\cap (\T*{_{\rm epi.}}\U)\subseteq \U$.
\end{thm}

\begin{proof}
We show the ``if" part first. By \cite[Theorem 2.32]{LN}, we know that $\underline \h$ is semi-abelian. By the result in \cite{R}, it is left integral if and only if it is right integral. We show that $\underline \h$ is left integral.

Let $\underline d: C\to D$ be an epimorphism in $\underline \h$ and $\underline b:B\to D$ be any morphism in $\underline \h$. Note that the following diagram
$$\xymatrix{
A \ar[r]^{\underline a} \ar[d]_{\underline c} &B \ar[d]^{\underline b}\\
C \ar[r]_{\underline d} &D
}
$$
in $\underline \h$ is a pull-back diagram if and only if $A\xrightarrow{\svecv{\underline a}{\underline c}} B\oplus C$ is the kernel of $B\oplus C\xrightarrow{\svech{\underline b}{-\underline d}} D$.

Since $\underline d$ is an epimorphism, it admits an epi-$\EE$-triangle $C\xrightarrow{d} D\to U\dashrightarrow$. $D$ admits an $\EE$-triangle $V_D\to W_D \xrightarrow{w} D\dashrightarrow$ with $V_D\in \V$, $W_D\in \W$. Then we have the following commutative diagram
$$\xymatrix{
V_D \ar[r] \ar@{=}[d] &X \ar[r] \ar[d] &B \ar[d]^b \ar@{-->}[r] &\\
V_D \ar[r] &W_D \ar[r]_w &D\ar@{-->}[r] &
}
$$
which induces an $\EE$-triangle $X\longrightarrow B\oplus W_D \xrightarrow{\svech{b}{w}} D\dashrightarrow$. Then we have the following commutative diagram.
$$\xymatrix{
X \ar@{=}[r] \ar[d] &X \ar[d]\\
A_1 \ar[r]^-{\svecv{b_1}{w_1}} \ar[d]_{c_1} &B\oplus W_D \ar[r] \ar[d]^{\svech{b}{w}} &U \ar@{=}[d] \ar@{-->}[r] &\\
C \ar[r]_d \ar@{-->}[d] &D \ar[r] \ar@{-->}[d] &U \ar@{-->}[r] &\\
&&
}$$
By Lemma \ref{L4}, there exists a morphism $\svech{b'}{w'}:B\oplus W_D\to D$ and a commutative diagram
$$\xymatrix{
A_1 \ar[r]^-{\svecv{b_1}{w_1}} \ar[d]_{c_1} &B\oplus W_D \ar[r] \ar[d]^{\svech{b'}{w'}} &U \ar@{=}[d] \ar@{-->}[r] &\\
C \ar[r]_d &D \ar[r] &U \ar@{-->}[r] &\\
}$$
which induces an $\EE$-triangle $A_1 \xrightarrow{\left(\begin{smallmatrix}b_1\\c_1\\w_1\end{smallmatrix}\right)} B\oplus C\oplus W_D \xrightarrow{\left(\begin{smallmatrix}b'~&-d&~w'\end{smallmatrix}\right)} D\dashrightarrow$.
We have that $\svech{b}{w}-\svech{b'}{w'}$ factors through $U$, since any morphism $U$ to $D$ factors through $w$, we get that $\underline b=\underline b'$.

On the other hand, there exists a morphism $\svecv{\widetilde{b}_1}{\widetilde{w}_1}:A_1\to B\oplus W_D$ and a commutative diagram
$$\xymatrix{
X \ar[r] \ar@{=}[d] &A_1 \ar[r]^{c_1} \ar[d]^{\svecv{\widetilde{b}_1}{\widetilde{w}_1}} &C \ar[d]^d \ar@{-->}[r] &\\
X \ar[r] &B\oplus W_D \ar[r]_-{\svech{b}{w}} &D \ar@{-->}[r] &
}$$
which induces an $\EE$-triangle $A_1 \xrightarrow{\left(\begin{smallmatrix}\widetilde{b}_1\\c_1\\\widetilde{w}_1\end{smallmatrix}\right)} B\oplus C\oplus W_D \xrightarrow{\left(\begin{smallmatrix}b&-d&w\end{smallmatrix}\right)} D\dashrightarrow$. Since $w$ is $\W$-epic, $\left(\begin{smallmatrix}b~&-d&~w\end{smallmatrix}\right)$ is also $\W$-epic.  By the dual of Lemma \ref{L2}, $A_1\in\B^+$. Since $\svecv{b_1}{w_1}-\svecv{\widetilde{b}_1}{\widetilde{w}_1}$ factors through $c_1$, there exists a morphism $\svecv{x}{y}: C\to B\oplus W_D$ such that $\svecv{b_1}{w_1}-\svecv{\widetilde{b}_1}{\widetilde{w}_1}=\svecv{x}{y}\circ c_1$. Then we have the following commutative diagram
$$\xymatrix{
A_1 \ar[r]^-{\left(\begin{smallmatrix}\widetilde{b}_1\\c_1\\\widetilde{w}_1\end{smallmatrix}\right)} \ar@{=}[d] &B\oplus C\oplus W_D \ar[d]^-{\left(\begin{smallmatrix}1&x&0\\0&1&0\\0&y&1\end{smallmatrix}\right)}_-{\cong} \ar[rr]^-{\left(\begin{smallmatrix}b&-d&w\end{smallmatrix}\right)} &&D\ar@{-->}[r] \ar[d]^z &\\
A_1 \ar[r]_-{\left(\begin{smallmatrix} b_1\\c_1\\w_1\end{smallmatrix}\right)} &B\oplus C\oplus W_D \ar[rr]_-{\left(\begin{smallmatrix}b'&-d&w'\end{smallmatrix}\right)} &&D \ar@{-->}[r] &
}
$$
which implies that $z$ is an isomorphism.
We also have another commutative diagram
$$\xymatrix{
V_D \ar[r] \ar@{=}[d] &A_1' \ar[r]^{\svecv{b_1'}{c_1'}\quad} \ar[d]^{w_1'} &B\oplus C \ar[d]^{\svech{-b}{d}} \ar@{-->}[r] &\\
V_D \ar[r] &W_D \ar[r]_w &D\ar@{-->}[r] &
}
$$
which induces an $\EE$-triangle $A_1'\xrightarrow{\left(\begin{smallmatrix}b_1'\\c_1'\\w_1'\end{smallmatrix}\right)} B\oplus C\oplus W_D \xrightarrow{\left(\begin{smallmatrix}b&-d&w\end{smallmatrix}\right)} D\dashrightarrow$. Then we have the following commutative diagram
$$\xymatrix{
A_1' \ar[r]^-{\left(\begin{smallmatrix}b_1'\\c_1\\w_1'\end{smallmatrix}\right)} \ar[d]_{\cong} &B\oplus C\oplus W_D \ar@{=}[d] \ar[rr]^-{\left(\begin{smallmatrix}b&-d&w\end{smallmatrix}\right)} &&D\ar@{-->}[r] \ar@{=}[d] &\\
A_1 \ar[r]_-{\left(\begin{smallmatrix} \widetilde{b}_1\\c_1\\ \widetilde{w}_1\end{smallmatrix}\right)} &B\oplus C\oplus W_D \ar[rr]_-{\left(\begin{smallmatrix}b&-d&w\end{smallmatrix}\right)} &&D \ar@{-->}[r] &.
}
$$
Hence we can get an isomorphism between $\EE$-triangles
$$\xymatrix{
A_1' \ar[r]^-{\left(\begin{smallmatrix}b_1'\\c_1\\w_1'\end{smallmatrix}\right)} \ar[d]_{\cong} &B\oplus C\oplus W_D \ar[d]^{\cong} \ar[rr]^-{\left(\begin{smallmatrix}b&-d&w\end{smallmatrix}\right)} &&D\ar@{-->}[r] \ar[d]^{\cong} &\\
A_1 \ar[r]_-{\left(\begin{smallmatrix} b_1\\c_1\\ w_1\end{smallmatrix}\right)} &B\oplus C\oplus W_D \ar[rr]_-{\left(\begin{smallmatrix}b'&-d&w'\end{smallmatrix}\right)} &&D \ar@{-->}[r] &.
}
$$
Note that $A_1$ admits a commutative diagram
$$\xymatrix{
V_1 \ar@{=}[r] \ar[d] &V_1 \ar[d]\\
A \ar[r]^{w_0\quad} \ar[d]_{a_1} &W_1 \ar[r] \ar[d] &S_1 \ar@{-->}[r] \ar@{=}[d] &\\
A_1 \ar[r] \ar@{-->}[d] &T_1 \ar[r] \ar@{-->}[d] &S_1 \ar@{-->}[r] &\\
&&
}
$$
with $T_1\in \T$, $S_1\in \s$, $V_1\in \V$, $W_1\in \W$ and $A\in \h$. By Lemma \ref{unique}, $\svecv{\underline {b_1a_1}}{\underline {c_1a_1}}:A\to B\oplus C$ is the kernel of $\svech{\underline b'}{-\underline d}=\svech{\underline b}{-\underline d}$.  Now we only need to show that $\underline {b_1a_1}$ is an epimorphism.
The diagram above induces an $\EE$-triangle $$A\xrightarrow{\svecv{a_1}{w_0}} A_1\oplus W_1\longrightarrow T_1\dashrightarrow.$$ Then we have the following commutative diagram.
$$\xymatrix{
A \ar@{=}[r] \ar[d]_-{\svecv{a_1}{w_0}} &A \ar[d]\\
A_1\oplus W_1 \ar[r]^-{\left(\begin{smallmatrix}b_1&0\\w_1&0\\0&1\end{smallmatrix} \right)} \ar[d] &B\oplus W_D\oplus W_1 \ar[r] \ar[d] &U \ar@{=}[d] \ar@{-->}[r] &\\
T_1 \ar[r] \ar@{-->}[d] &Y \ar[r] \ar@{-->}[d] &U \ar@{-->}[r] &\\
&&
}
$$
Since $A$ admits an $\EE$-triangle $A\to W^A\to S^A\dashrightarrow$ with $W^A\in \W$ and $S^A\in \s$, we have the following commutative diagram
$$\xymatrix{
A \ar[r] \ar[d]_-{\left(\begin{smallmatrix}b_1a_1\\w_1a_1\\w_0\end{smallmatrix} \right)} &W^A \ar[r] \ar[d] &S^A \ar@{=}[d] \ar@{-->}[r] &\\
B\oplus W_D\oplus W_1 \ar[r] \ar[d] &Z \ar[r] \ar[d] &S^A \ar@{-->}[r] &\\
Y \ar@{-->}[d] \ar@{=}[r] &Y \ar@{-->}[d]\\
&&
}$$
with $Z\in \B^-$. It is enough to show that $Z\in \U$. Since we have the following commutative diagram
$$\xymatrix{
W^A \ar@{=}[r] \ar[d] &W^A \ar[d]\\
T_2 \ar[r] \ar[d] &Z \ar[r] \ar[d] &U \ar@{=}[d] \ar@{-->}[r] &\\
T_1 \ar[r] \ar@{-->}[d] &Y \ar[r] \ar@{-->}[d] &U \ar@{-->}[r] &\\
&&
}
$$
with $T_2\in \T$, we obtain that $Z\in \B^-\cap (\T*{_{\rm epi.}}\U)\subseteq \U$.

\medskip

Now we show the `` only if " part. Let $X\in \B^-\cap (\T*{_{\rm epi.}}\U)$. It admits an $\EE$-triangle $T\to X\to U\dashrightarrow$ with $T\in \T$ and $U\in {_{\rm epi.}}\U$. $U$ admits an epi-$\EE$-triangle $C\xrightarrow{d} D\to U\dashrightarrow$. Then we have the following commutative diagram
$$\xymatrix{
&T \ar@{=}[r] \ar[d] &T \ar[d]\\
C \ar[r]^r \ar@{=}[d] &B \ar[r] \ar[d]^b &X\ar[d] \ar@{-->}[r] &\\
C \ar[r]_d  &D \ar[r] \ar@{-->}[d] &U \ar@{-->}[d] \ar@{-->}[r] &\\
&&
}
$$
From the second column we can get that $B\in \B^+$. Since $d$ is $\W$-monic, we can get that $r$ is also $\W$-monic. By Lemma \ref{L2}, $B\in \B^-$. Hence $B\in \h$. Consider a pull-back diagram in $\underline \h$:
$$\xymatrix{
A \ar[r]^{\underline a} \ar[d]_{\underline c} &B \ar[d]^{\underline b}\\
C \ar[r]_{\underline d} &D
}
$$
Since $\underline \h$ is integral, $\underline a$ is an epimorphism. Then $\underline {dc}=\underline {brc}=\underline {ba}$. Since $\underline b$ is a monomorphism by Lemma \ref{+-}, we have $\underline {rc}=\underline a$. Hence $\underline r$ is an epimorphism. By Lemma \ref{L3}, we have $X\in \U$. Thus $\B^-\cap (\T*{_{\rm epi.}}\U)\subseteq \U$.
\end{proof}

By duality, we have the following corollary.

\begin{cor}
$\underline \h$ is integral if and only if $\B^+\cap (\T_{\rm mono.}*\U)\subseteq \T$.
\end{cor}

The following example shows that the hearts are not always integral.

\begin{exm}\label{ex2}
{\upshape 	 Let $A=kQ/I$ be an algebra given by the quiver
	 \begin{align}
	 	\begin{minipage}{0.6\hsize}
	 		\ \ \ \ \  \xymatrix{Q: \begin{smallmatrix}6\end{smallmatrix}\ar[r]^{\alpha}&\begin{smallmatrix}5\end{smallmatrix}\ar[r]^{\beta}
	 			&\begin{smallmatrix}4\end{smallmatrix}\ar[r]^{\gamma}&\begin{smallmatrix}3\end{smallmatrix}\ar[r]^{\delta}&\begin{smallmatrix}2\end{smallmatrix}\ar[r]^{\varepsilon}&\begin{smallmatrix}1\end{smallmatrix}}\notag
	 	\end{minipage}
	 \end{align}
and $I=\langle\alpha\beta\gamma\delta, \beta\gamma\delta\varepsilon\rangle.$	The Auslander-Reiten quiver of $\mod A$ is the following:}
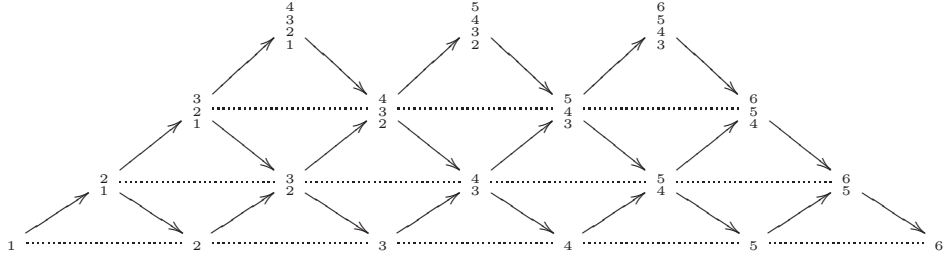
\begin{align}
	\tiny{\xymatrix @R=4mm @C=8mm{&&&
			{\begin{smallmatrix}4\\3\\2\\1\end{smallmatrix}}\ar[rd]&&
			{\begin{smallmatrix}5\\4\\3\\2\end{smallmatrix}}\ar[rd]&&
			{\begin{smallmatrix}6\\5\\4\\3\end{smallmatrix}}\ar[rd]&&&\\
			&&{\begin{smallmatrix}3\\2\\1\end{smallmatrix}}\ar[ru]\ar[rd]&&
			{\begin{smallmatrix}4\\3\\2\end{smallmatrix}}\ar[ru]\ar[rd]\ar@{.}[ll]&&
			{\begin{smallmatrix}5\\4\\3\end{smallmatrix}}\ar[ru]\ar[rd]\ar@{.}[ll]&&
			{\begin{smallmatrix}6\\5\\4\end{smallmatrix}}\ar[rd]\ar@{.}[ll]&&\\
			&{\begin{smallmatrix}2\\1\end{smallmatrix}}\ar[ru]\ar[rd]&&
			{\begin{smallmatrix}3\\2\end{smallmatrix}}\ar[ru]\ar[rd]\ar@{.}[ll]&&
			{\begin{smallmatrix}4\\3\end{smallmatrix}}\ar[ru]\ar[rd]\ar@{.}[ll]&&
			{\begin{smallmatrix}5\\4\end{smallmatrix}}\ar[ru]\ar[rd]\ar@{.}[ll]&&
			{\begin{smallmatrix}6\\5\end{smallmatrix}}\ar[rd]\ar@{.}[ll]\\
			{\begin{smallmatrix}1\end{smallmatrix}}\ar[ru]&&\textnormal{$\begin{smallmatrix}2\end{smallmatrix}$}\ar[ru]\ar@{.}[ll]&&
			{\begin{smallmatrix}3\end{smallmatrix}}\ar[ru]\ar@{.}[ll]&&\textnormal{$\begin{smallmatrix}4\end{smallmatrix}$}\ar[ru]\ar@{.}[ll]&&
			{\begin{smallmatrix}5\end{smallmatrix}}\ar[ru]\ar@{.}[ll]&&
			{\begin{smallmatrix}6\end{smallmatrix}}\ar@{.}[ll]
	}}\notag
\end{align}
\end{exm}
In this example, we denote by ``$\bullet$" in the quiver the objects belong to a subcategory and
by  ``$\circ $" the objects do not. Given four subcategories as follows:
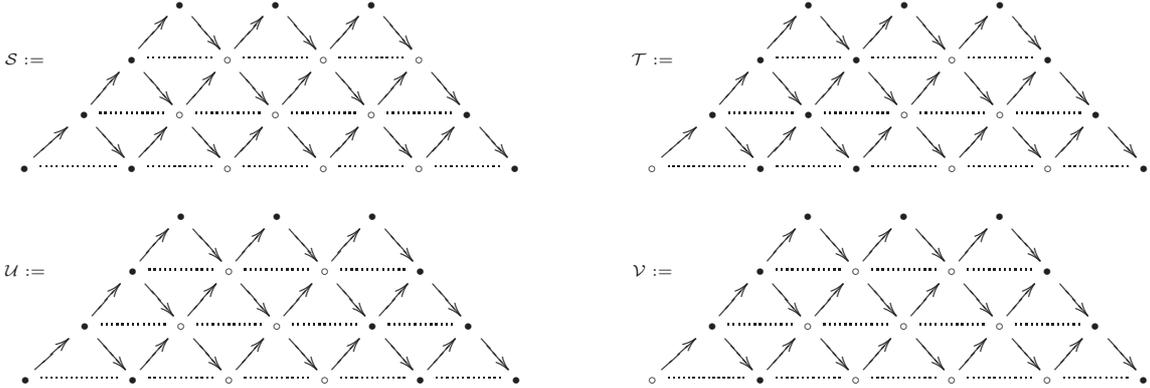
\begin{align}
	\tiny{\xymatrix @R=4mm @C2mm{&&&
			{\begin{smallmatrix}\bullet\end{smallmatrix}}\ar[rd]&&
			{\begin{smallmatrix}\bullet\end{smallmatrix}}\ar[rd]&&
			{\begin{smallmatrix}\bullet\end{smallmatrix}}\ar[rd]&&&\\
			\s:=&&{\begin{smallmatrix}\bullet\end{smallmatrix}}\ar[ru]\ar[rd]&&
			{\begin{smallmatrix}\circ\end{smallmatrix}}\ar[ru]\ar[rd]\ar@{.}[ll]&&
			{\begin{smallmatrix}\circ\end{smallmatrix}}\ar[ru]\ar[rd]\ar@{.}[ll]&&
			{\begin{smallmatrix}\circ\end{smallmatrix}}\ar[rd]\ar@{.}[ll]&&\\
			&{\begin{smallmatrix}\bullet\end{smallmatrix}}\ar[ru]\ar[rd]&&
			{\begin{smallmatrix}\circ\end{smallmatrix}}\ar[ru]\ar[rd]\ar@{.}[ll]&&
			{\begin{smallmatrix}\circ\end{smallmatrix}}\ar[ru]\ar[rd]\ar@{.}[ll]&&
			{\begin{smallmatrix}\circ\end{smallmatrix}}\ar[ru]\ar[rd]\ar@{.}[ll]&&
			{\begin{smallmatrix}\bullet\end{smallmatrix}}\ar[rd]\ar@{.}[ll]\\
			{\begin{smallmatrix}\bullet\end{smallmatrix}}\ar[ru]&&\textnormal{$\begin{smallmatrix}\bullet\end{smallmatrix}$}\ar[ru]\ar@{.}[ll]&&
			{\begin{smallmatrix}\circ\end{smallmatrix}}\ar[ru]\ar@{.}[ll]&&\textnormal{$\begin{smallmatrix}\circ\end{smallmatrix}$}\ar[ru]\ar@{.}[ll]&&
			{\begin{smallmatrix}\circ\end{smallmatrix}}\ar[ru]\ar@{.}[ll]&&
			{\begin{smallmatrix}\bullet\end{smallmatrix}}\ar@{.}[ll]\\
	}}\notag
\ \ \quad\quad \quad\quad	\tiny{\xymatrix @R=4mm @C2mm{&&&
			{\begin{smallmatrix}\bullet\end{smallmatrix}}\ar[rd]&&
			{\begin{smallmatrix}\bullet\end{smallmatrix}}\ar[rd]&&
			{\begin{smallmatrix}\bullet\end{smallmatrix}}\ar[rd]&&&\\
			\T:=&&{\begin{smallmatrix}\bullet\end{smallmatrix}}\ar[ru]\ar[rd]&&
			{\begin{smallmatrix}\bullet\end{smallmatrix}}\ar[ru]\ar[rd]\ar@{.}[ll]&&
			{\begin{smallmatrix}\circ\end{smallmatrix}}\ar[ru]\ar[rd]\ar@{.}[ll]&&
			{\begin{smallmatrix}\bullet\end{smallmatrix}}\ar[rd]\ar@{.}[ll]&&\\
			&{\begin{smallmatrix}\bullet\end{smallmatrix}}\ar[ru]\ar[rd]&&
			{\begin{smallmatrix}\bullet\end{smallmatrix}}\ar[ru]\ar[rd]\ar@{.}[ll]&&
			{\begin{smallmatrix}\circ\end{smallmatrix}}\ar[ru]\ar[rd]\ar@{.}[ll]&&
			{\begin{smallmatrix}\circ\end{smallmatrix}}\ar[ru]\ar[rd]\ar@{.}[ll]&&
			{\begin{smallmatrix}\bullet\end{smallmatrix}}\ar[rd]\ar@{.}[ll]\\
			{\begin{smallmatrix}\circ\end{smallmatrix}}\ar[ru]&&\textnormal{$\begin{smallmatrix}\bullet\end{smallmatrix}$}\ar[ru]\ar@{.}[ll]&&
			{\begin{smallmatrix}\bullet\end{smallmatrix}}\ar[ru]\ar@{.}[ll]&&\textnormal{$\begin{smallmatrix}\circ\end{smallmatrix}$}\ar[ru]\ar@{.}[ll]&&
			{\begin{smallmatrix}\circ\end{smallmatrix}}\ar[ru]\ar@{.}[ll]&&
			{\begin{smallmatrix}\bullet\end{smallmatrix}}\ar@{.}[ll]\\
	}}\notag
\end{align}
\begin{align}
	\tiny{\xymatrix @R=4mm @C2mm{&&&
			{\begin{smallmatrix}\bullet\end{smallmatrix}}\ar[rd]&&
			{\begin{smallmatrix}\bullet\end{smallmatrix}}\ar[rd]&&
			{\begin{smallmatrix}\bullet\end{smallmatrix}}\ar[rd]&&&\\
			\U:=&&{\begin{smallmatrix}\bullet\end{smallmatrix}}\ar[ru]\ar[rd]&&
			{\begin{smallmatrix}\circ\end{smallmatrix}}\ar[ru]\ar[rd]\ar@{.}[ll]&&
			{\begin{smallmatrix}\circ\end{smallmatrix}}\ar[ru]\ar[rd]\ar@{.}[ll]&&
			{\begin{smallmatrix}\bullet\end{smallmatrix}}\ar[rd]\ar@{.}[ll]&&\\
			&{\begin{smallmatrix}\bullet\end{smallmatrix}}\ar[ru]\ar[rd]&&
			{\begin{smallmatrix}\circ\end{smallmatrix}}\ar[ru]\ar[rd]\ar@{.}[ll]&&
			{\begin{smallmatrix}\circ\end{smallmatrix}}\ar[ru]\ar[rd]\ar@{.}[ll]&&
			{\begin{smallmatrix}\bullet\end{smallmatrix}}\ar[ru]\ar[rd]\ar@{.}[ll]&&
			{\begin{smallmatrix}\bullet\end{smallmatrix}}\ar[rd]\ar@{.}[ll]\\
			{\begin{smallmatrix}\bullet\end{smallmatrix}}\ar[ru]&&\textnormal{$\begin{smallmatrix}\bullet\end{smallmatrix}$}\ar[ru]\ar@{.}[ll]&&
			{\begin{smallmatrix}\circ\end{smallmatrix}}\ar[ru]\ar@{.}[ll]&&\textnormal{$\begin{smallmatrix}\circ\end{smallmatrix}$}\ar[ru]\ar@{.}[ll]&&
			{\begin{smallmatrix}\bullet\end{smallmatrix}}\ar[ru]\ar@{.}[ll]&&
			{\begin{smallmatrix}\bullet\end{smallmatrix}}\ar@{.}[ll]\\
	}}\notag
	\ \ \quad\quad \quad\quad	\tiny{\xymatrix @R=4mm @C2mm{&&&
		{\begin{smallmatrix}\bullet\end{smallmatrix}}\ar[rd]&&
			{\begin{smallmatrix}\bullet\end{smallmatrix}}\ar[rd]&&
			{\begin{smallmatrix}\bullet\end{smallmatrix}}\ar[rd]&&&\\
			\V:=&&{\begin{smallmatrix}\bullet\end{smallmatrix}}\ar[ru]\ar[rd]&&
			{\begin{smallmatrix}\circ\end{smallmatrix}}\ar[ru]\ar[rd]\ar@{.}[ll]&&
			{\begin{smallmatrix}\circ\end{smallmatrix}}\ar[ru]\ar[rd]\ar@{.}[ll]&&
			{\begin{smallmatrix}\bullet\end{smallmatrix}}\ar[rd]\ar@{.}[ll]&&\\
			&{\begin{smallmatrix}\bullet\end{smallmatrix}}\ar[ru]\ar[rd]&&
			{\begin{smallmatrix}\circ\end{smallmatrix}}\ar[ru]\ar[rd]\ar@{.}[ll]&&
			{\begin{smallmatrix}\circ\end{smallmatrix}}\ar[ru]\ar[rd]\ar@{.}[ll]&&
			{\begin{smallmatrix}\circ\end{smallmatrix}}\ar[ru]\ar[rd]\ar@{.}[ll]&&
			{\begin{smallmatrix}\bullet\end{smallmatrix}}\ar[rd]\ar@{.}[ll]\\
			{\begin{smallmatrix}\circ\end{smallmatrix}}\ar[ru]&&\textnormal{$\begin{smallmatrix}\bullet\end{smallmatrix}$}\ar[ru]\ar@{.}[ll]&&
			{\begin{smallmatrix}\circ\end{smallmatrix}}\ar[ru]\ar@{.}[ll]&&\textnormal{$\begin{smallmatrix}\circ\end{smallmatrix}$}\ar[ru]\ar@{.}[ll]&&
			{\begin{smallmatrix}\circ\end{smallmatrix}}\ar[ru]\ar@{.}[ll]&&
			{\begin{smallmatrix}\bullet\end{smallmatrix}}\ar@{.}[ll]\\
	}}\notag
\end{align}
then $(\s, \T), (\U, \V)$ form a twin cotorsion pair and $\underline \h=\add\bigg\{ \begin{smallmatrix}4\\3\end{smallmatrix}\oplus\begin{smallmatrix}5\\4\\3\end{smallmatrix}\oplus\begin{smallmatrix}4\end{smallmatrix} \bigg\}$. This heart is not integral, since we have short exact sequences
$$\xymatrix{
{\begin{smallmatrix}4\\3\end{smallmatrix}} \ar[r] &{\begin{smallmatrix}5\\4\\3\end{smallmatrix}}\oplus\begin{smallmatrix}4\end{smallmatrix} \ar[r] &{\begin{smallmatrix}5\\4\end{smallmatrix}},} \quad \xymatrix{{\begin{smallmatrix}3\end{smallmatrix}} \ar[r] &{\begin{smallmatrix}5\\4\\3\end{smallmatrix}} \ar[r] &{\begin{smallmatrix}5\\4\end{smallmatrix}}}
$$
with ${\begin{smallmatrix}5\\4\end{smallmatrix}}\in {_{epi.}}\U$ and ${\begin{smallmatrix}3\end{smallmatrix}}\in \T$.

\medskip

The following proposition shows that the condition $\U\subseteq (\s*\T)$ (resp. $\T\subseteq (\U*\V)$) itself is a sufficient condition for hearts being abelian (there is no need to assume that $\B$ has enough projectives or enough injectives), hence generalizes Theorem \ref{1}.

\begin{prop}\label{gen}
For a twin cotorsion pair $((\s,\T),(\U,\V))$, we have:
\begin{itemize}
\item[(1)] if $\U\subseteq (\s*\T)$, then $\B^-\cap (\T*\U)\subseteq \U$;
\item[(2)] if $\T\subseteq (\U*\V)$, then $\B^+\cap (\T*\U)\subseteq \T$.
\end{itemize}
\end{prop}

\begin{proof}
We show that $\U\subseteq (\s*\T)$ implies $\B^-\cap (\T* \U)\subseteq \U$, the other case is by dual.

Assume we have an $\EE$-triangle $T\to X\to U\dashrightarrow$ with $T\in \T$, $X\in \B^-$ and $U\in \U$. Since $\U\subseteq (\s*\T)$, $U$ admits an $\EE$-triangle $S_1\to U\to T_1\dashrightarrow$ with $S_1\in \s$ and $T_1\in \T$. Then we have the following commutative diagrams
$$\xymatrix{
T \ar@{=}[r] \ar[d] &T \ar[d]\\
S_1 \oplus T \ar[r] \ar[d] &X \ar[r] \ar[d] &T_1 \ar@{=}[d] \ar@{-->}[r] &\\
S_1 \ar[r] \ar@{-->}[d] &U \ar[r] \ar@{-->}[d] &T_1 \ar@{-->}[r] &\\
&&\\
}\quad \quad \xymatrix{
S_1 \ar@{=}[r] \ar[d] &S_1 \ar[d]\\
S_1 \oplus T \ar[r] \ar[d] &X \ar[r] \ar[d] &T_1 \ar@{=}[d] \ar@{-->}[r] &\\
T \ar[r] \ar@{-->}[d] &T_2 \ar[r] \ar@{-->}[d] &T_1 \ar@{-->}[r] &\\
&&
}
$$
with $T_2\in \T$. Since $X$ admits an $\EE$-triangle $X\xrightarrow{w} W^X\to S^X\dashrightarrow$ with $W^X\in \W$ and $S^X\in \s$, we have the following commutative diagrams
$$\xymatrix{
S_1 \ar[r] \ar@{=}[d] &X \ar[r]^t \ar[d]^w &T_2 \ar@{-->}[r] \ar[d]^{\svecv{1}{0}} &\\
S_1 \ar[r] &W^X \ar[r]^-{\svecv{w_1}{w_2}} \ar[d] &T_2\oplus S^X \ar@{-->}[r] \ar[d] &\\
&S^X \ar@{=}[r] \ar@{-->}[d] &S^X \ar@{-->}[d]\\
&&\\
}\quad\quad \xymatrix{
S_1 \ar[r] \ar@{=}[d] &S_2 \ar[r] \ar[d] &S^X \ar@{-->}[r] \ar[d] &\\
S_1 \ar[r] &W^X \ar[r]^-{\svecv{w_1}{w_2}} \ar[d]^{w_1} &T_2\oplus S^X \ar@{-->}[r] \ar[d]^{\svech{1}{0}} &\\
&T_2 \ar@{=}[r] \ar@{-->}[d] &T_2 \ar@{-->}[d]\\
&&
}
$$
with $S_2\in \s$. Then we have the following commutative diagram
$$\xymatrix{
S_1 \ar[r] \ar[d] &X \ar[r]^t \ar[d]^w &T_2 \ar@{-->}[r] \ar@{=}[d] &\\
S_2 \ar[r] &W_X \ar[r]_{w_1} &T_2 \ar@{-->}[r] &
}
$$
which induces an $\EE$-triangle $S_1\to X\oplus S_2 \to W_X\dashrightarrow$. Hence $X\oplus S_2\in \U$, which implies that $X\in \U$.
\end{proof}

The following example shows that the conditions in Proposition \ref{gen} are only sufficient for the hearts being abelian.

\begin{exm}\label{ex1}
	Let $A=kQ/I$ be the algebra given in Example \ref{ex2}. We still denote by ``$\bullet$" in the quiver the objects belong to a subcategory and
by  ``$\circ $" the objects do not. Given four subcategories as follows:
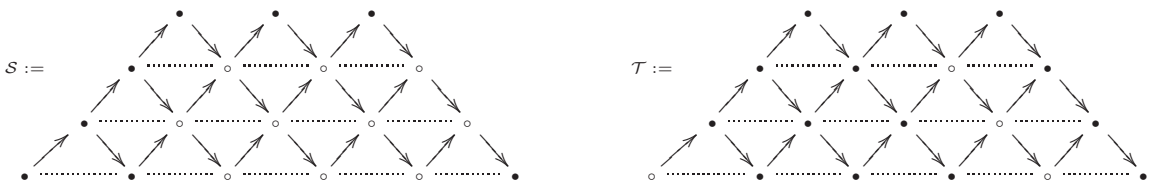
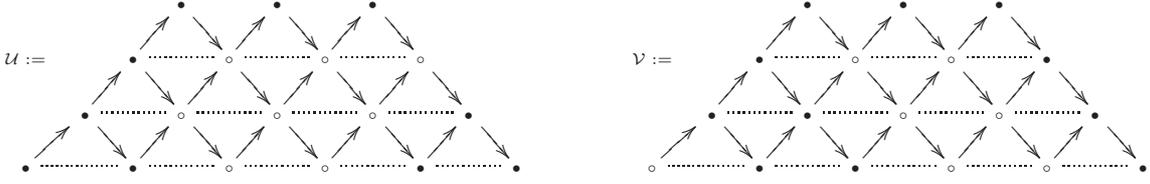
\begin{align}
	\tiny{\xymatrix @R=4mm @C2mm{&&&
			{\begin{smallmatrix}\bullet\end{smallmatrix}}\ar[rd]&&
			{\begin{smallmatrix}\bullet\end{smallmatrix}}\ar[rd]&&
			{\begin{smallmatrix}\bullet\end{smallmatrix}}\ar[rd]&&&\\
			\s:=&&{\begin{smallmatrix}\bullet\end{smallmatrix}}\ar[ru]\ar[rd]&&
			{\begin{smallmatrix}\circ\end{smallmatrix}}\ar[ru]\ar[rd]\ar@{.}[ll]&&
			{\begin{smallmatrix}\circ\end{smallmatrix}}\ar[ru]\ar[rd]\ar@{.}[ll]&&
			{\begin{smallmatrix}\circ\end{smallmatrix}}\ar[rd]\ar@{.}[ll]&&\\
			&{\begin{smallmatrix}\bullet\end{smallmatrix}}\ar[ru]\ar[rd]&&
			{\begin{smallmatrix}\circ\end{smallmatrix}}\ar[ru]\ar[rd]\ar@{.}[ll]&&
			{\begin{smallmatrix}\circ\end{smallmatrix}}\ar[ru]\ar[rd]\ar@{.}[ll]&&
			{\begin{smallmatrix}\circ\end{smallmatrix}}\ar[ru]\ar[rd]\ar@{.}[ll]&&
			{\begin{smallmatrix}\circ\end{smallmatrix}}\ar[rd]\ar@{.}[ll]\\
			{\begin{smallmatrix}\bullet\end{smallmatrix}}\ar[ru]&&\textnormal{$\begin{smallmatrix}\bullet\end{smallmatrix}$}\ar[ru]\ar@{.}[ll]&&
			{\begin{smallmatrix}\circ\end{smallmatrix}}\ar[ru]\ar@{.}[ll]&&\textnormal{$\begin{smallmatrix}\circ\end{smallmatrix}$}\ar[ru]\ar@{.}[ll]&&
			{\begin{smallmatrix}\circ\end{smallmatrix}}\ar[ru]\ar@{.}[ll]&&
			{\begin{smallmatrix}\bullet\end{smallmatrix}}\ar@{.}[ll]\\
	}}\notag
\ \ \quad\quad \quad\quad	\tiny{\xymatrix @R=4mm @C2mm{&&&
			{\begin{smallmatrix}\bullet\end{smallmatrix}}\ar[rd]&&
			{\begin{smallmatrix}\bullet\end{smallmatrix}}\ar[rd]&&
			{\begin{smallmatrix}\bullet\end{smallmatrix}}\ar[rd]&&&\\
			\T:=&&{\begin{smallmatrix}\bullet\end{smallmatrix}}\ar[ru]\ar[rd]&&
			{\begin{smallmatrix}\bullet\end{smallmatrix}}\ar[ru]\ar[rd]\ar@{.}[ll]&&
			{\begin{smallmatrix}\circ\end{smallmatrix}}\ar[ru]\ar[rd]\ar@{.}[ll]&&
			{\begin{smallmatrix}\bullet\end{smallmatrix}}\ar[rd]\ar@{.}[ll]&&\\
			&{\begin{smallmatrix}\bullet\end{smallmatrix}}\ar[ru]\ar[rd]&&
			{\begin{smallmatrix}\bullet\end{smallmatrix}}\ar[ru]\ar[rd]\ar@{.}[ll]&&
			{\begin{smallmatrix}\bullet\end{smallmatrix}}\ar[ru]\ar[rd]\ar@{.}[ll]&&
			{\begin{smallmatrix}\circ\end{smallmatrix}}\ar[ru]\ar[rd]\ar@{.}[ll]&&
			{\begin{smallmatrix}\bullet\end{smallmatrix}}\ar[rd]\ar@{.}[ll]\\
			{\begin{smallmatrix}\circ\end{smallmatrix}}\ar[ru]&&\textnormal{$\begin{smallmatrix}\bullet\end{smallmatrix}$}\ar[ru]\ar@{.}[ll]&&
			{\begin{smallmatrix}\bullet\end{smallmatrix}}\ar[ru]\ar@{.}[ll]&&\textnormal{$\begin{smallmatrix}\bullet\end{smallmatrix}$}\ar[ru]\ar@{.}[ll]&&
			{\begin{smallmatrix}\circ\end{smallmatrix}}\ar[ru]\ar@{.}[ll]&&
			{\begin{smallmatrix}\bullet\end{smallmatrix}}\ar@{.}[ll]\\
	}}\notag
\end{align}
\begin{align}
	\tiny{\xymatrix @R=4mm @C2mm{&&&
			{\begin{smallmatrix}\bullet\end{smallmatrix}}\ar[rd]&&
			{\begin{smallmatrix}\bullet\end{smallmatrix}}\ar[rd]&&
			{\begin{smallmatrix}\bullet\end{smallmatrix}}\ar[rd]&&&\\
			\U:=&&{\begin{smallmatrix}\bullet\end{smallmatrix}}\ar[ru]\ar[rd]&&
			{\begin{smallmatrix}\circ\end{smallmatrix}}\ar[ru]\ar[rd]\ar@{.}[ll]&&
			{\begin{smallmatrix}\circ\end{smallmatrix}}\ar[ru]\ar[rd]\ar@{.}[ll]&&
			{\begin{smallmatrix}\circ\end{smallmatrix}}\ar[rd]\ar@{.}[ll]&&\\
			&{\begin{smallmatrix}\bullet\end{smallmatrix}}\ar[ru]\ar[rd]&&
			{\begin{smallmatrix}\circ\end{smallmatrix}}\ar[ru]\ar[rd]\ar@{.}[ll]&&
			{\begin{smallmatrix}\circ\end{smallmatrix}}\ar[ru]\ar[rd]\ar@{.}[ll]&&
			{\begin{smallmatrix}\circ\end{smallmatrix}}\ar[ru]\ar[rd]\ar@{.}[ll]&&
			{\begin{smallmatrix}\bullet\end{smallmatrix}}\ar[rd]\ar@{.}[ll]\\
			{\begin{smallmatrix}\bullet\end{smallmatrix}}\ar[ru]&&\textnormal{$\begin{smallmatrix}\bullet\end{smallmatrix}$}\ar[ru]\ar@{.}[ll]&&
			{\begin{smallmatrix}\circ\end{smallmatrix}}\ar[ru]\ar@{.}[ll]&&\textnormal{$\begin{smallmatrix}\circ\end{smallmatrix}$}\ar[ru]\ar@{.}[ll]&&
			{\begin{smallmatrix}\bullet\end{smallmatrix}}\ar[ru]\ar@{.}[ll]&&
			{\begin{smallmatrix}\bullet\end{smallmatrix}}\ar@{.}[ll]\\
	}}\notag
	\ \ \quad\quad \quad\quad	\tiny{\xymatrix @R=4mm @C2mm{&&&
		{\begin{smallmatrix}\bullet\end{smallmatrix}}\ar[rd]&&
			{\begin{smallmatrix}\bullet\end{smallmatrix}}\ar[rd]&&
			{\begin{smallmatrix}\bullet\end{smallmatrix}}\ar[rd]&&&\\
			\V:=&&{\begin{smallmatrix}\bullet\end{smallmatrix}}\ar[ru]\ar[rd]&&
			{\begin{smallmatrix}\circ\end{smallmatrix}}\ar[ru]\ar[rd]\ar@{.}[ll]&&
			{\begin{smallmatrix}\circ\end{smallmatrix}}\ar[ru]\ar[rd]\ar@{.}[ll]&&
			{\begin{smallmatrix}\bullet\end{smallmatrix}}\ar[rd]\ar@{.}[ll]&&\\
			&{\begin{smallmatrix}\bullet\end{smallmatrix}}\ar[ru]\ar[rd]&&
			{\begin{smallmatrix}\bullet\end{smallmatrix}}\ar[ru]\ar[rd]\ar@{.}[ll]&&
			{\begin{smallmatrix}\circ\end{smallmatrix}}\ar[ru]\ar[rd]\ar@{.}[ll]&&
			{\begin{smallmatrix}\circ\end{smallmatrix}}\ar[ru]\ar[rd]\ar@{.}[ll]&&
			{\begin{smallmatrix}\bullet\end{smallmatrix}}\ar[rd]\ar@{.}[ll]\\
			{\begin{smallmatrix}\circ\end{smallmatrix}}\ar[ru]&&\textnormal{$\begin{smallmatrix}\bullet\end{smallmatrix}$}\ar[ru]\ar@{.}[ll]&&
			{\begin{smallmatrix}\bullet\end{smallmatrix}}\ar[ru]\ar@{.}[ll]&&\textnormal{$\begin{smallmatrix}\circ\end{smallmatrix}$}\ar[ru]\ar@{.}[ll]&&
			{\begin{smallmatrix}\circ\end{smallmatrix}}\ar[ru]\ar@{.}[ll]&&
			{\begin{smallmatrix}\bullet\end{smallmatrix}}\ar@{.}[ll]\\
	}}\notag
\end{align}
then $((\s, \T), (\U, \V))$ forms a twin cotorsion pair and $\underline \h=\add\bigg\{\begin{smallmatrix}5\\4\\3\end{smallmatrix}\bigg\}$. This heart is abelian (hence is integral), but
$$\U \nsubseteq \s*\T, \quad \T\nsubseteq \U*\V.$$
\end{exm}

\section{Abelian Hearts}
In this section, we assume that $\B$ is Krull-Schmidt. We investigate several conditions when the hearts of twin cotorsion pairs being abelian, which are related to the results in the previous sections.

For convenience, we use the following notion introduced in \cite{Be}: Let $\X$ and $\Y$ be two full subcategories of $\B$. We denote by $\X\oplus\Y$ the full subcategory
of $\B$ consisting of the direct summands of direct sums $X\oplus Y$, where $X\in\X$ and $Y\in \Y$.

\begin{prop}\label{prop2}
If $\underline{\h}$ is abelian, then ${_{\rm epi.}}\U\subseteq\s\oplus\W$
and $\T_{\rm mono.}\subseteq\V\oplus\W$.
\end{prop}

\begin{proof}
We only show that ${_{\rm epi.}}\U\subseteq\s\oplus\W$, the other half is by dual. For any $U\in{_{\rm epi.}}\U$, there exists an epi-$\EE$-triangle $A\xrightarrow{~f~} B\longrightarrow U\dashrightarrow$. Since $(\s,\T)$ is a cotorsion pair on $\B$,
$U$ admits an $\EE$-triangle
$T\rightarrow S\to U\dashrightarrow$
where $T\in\T$ and $S\in\s$. We have the following commutative diagram
$$\xymatrix{
&T\ar@{=}[r] \ar[d] &T \ar[d]\\
A\ar[r]^c \ar@{=}[d]&C\ar[r] \ar[d]^g &S \ar[d] \ar@{-->}[r] &\\
A\ar[r]_f &B \ar[r] \ar@{-->}[d] &U \ar@{-->}[r] \ar@{-->}[d]&\\
&&
}$$
with $C\in \h$ by Lemma \ref{+-}. This second column shows that $\underline{g}$ is a monomorphism in $\underline{\h}$. Since $S$ admits an $\EE$-triangle $S\xrightarrow{m} M\to S'\dashrightarrow$ with $M\in \s\cap \T$ and $S'\in \s$, we have the following commutative diagram
$$\xymatrix{
C \ar[r]^{\svecv{g}{s}\quad} \ar@{=}[d] &B\oplus S \ar[r] \ar[d]^{ \left( \begin{smallmatrix}1&0\\0&m\end{smallmatrix} \right)} &U \ar@{-->}[r] \ar[d] &\\
C \ar[r]_{\svecv{g}{ms}\quad}  &B\oplus M \ar[r] \ar[d] &U' \ar@{-->}[r] \ar[d] &\\
&S' \ar@{=}[r] \ar@{-->}[d] &S' \ar@{-->}[d]\\
&&\\
}$$
with $U'\in \U$. This shows that $\underline{g}$ is an epimorphism in $\underline{\h}$.
Since $\underline{\h}$ is abelian,  $\underline{g}$ is an isomorphism. Let $\underline{g}^{-1}=\underline{h}$.
Since $\underline{h}\circ \underline{g}=\underline{1}$,
there exist two morphisms $w_1\colon C\to W$ and
$w_2\colon W\to C$ with $W\in \W$ such that $1-hg=w_2w_1$. It follows that $hf=h(gc)=c-w_2w_1c$. Since $f$ is $\W$-monic, there exists a morphism $w_3:B\to W$ such that $w_1c=w_3f$. Hence $c=(h+w_2w_3)f$. Let $h'=h+w_2w_3$. Then we have the following commutative diagram
$$\xymatrix{
A\ar[r]^f \ar@{=}[d]&B \ar[r] \ar[d]^{h'} &U \ar[d] \ar@{-->}[r] &\\
A \ar[r]^c \ar@{=}[d]&C \ar[r] \ar[d]^{g} &S \ar[d] \ar@{-->}[r] &\\
A \ar[r]^{f}  &B \ar[r] &U \ar@{-->}[r] &
}
$$
This implies that $U$ is a direct summand of $B\oplus S$. Since $B$ admits an $\EE$-triangle $V_B\to W_B\to B\dashrightarrow$ with $V_B\in \V$ and $W_B\in \W$, any morphism from $U$ to $B$ factors through $W_B$. Hence $U$ is a direct summand of $S\oplus W_B\in \s\oplus \W$.
\end{proof}

\begin{rem}
Since $\s\oplus \W\subseteq \s* \T$, by the proof of Proposition \ref{gen}, ${_{\rm epi.}}\U\subseteq\s\oplus\W$ implies that $\B^-\cap (\T*{_{\rm epi.}}\U)\subseteq \U$. Hence ${_{\rm epi.}}\U\subseteq\s\oplus\W$ is a sufficient condition for the heart of $((\s,\T),(\U,\V))$ being integral. By duality, so is the condition $\T_{\rm mono.}\subseteq\V\oplus\W$.
\end{rem}


Since $(\s,\T)$ and $(\U,\V)$ are cotorsion pairs, they have abelian hearts by \cite[Theorem 3.2]{LN}. For convenience, we introduce the following notations:
\begin{itemize}
\item[(a1)] $\B^+_1=\{X\in \B \text{ }|\text{ } \mbox{there exists an}~\text{ } \EE\text{-triangle } T\to M\to X\dashrightarrow \text{, }T\in \T \text{ and }M\in \s\cap\T \}$,
\item[(a2)] $\B^-_1=\{Y\in \B \text{ }|\text{ } \mbox{there exists an}~ \text{ } \EE\text{-triangle } Y\to M'\to S\dashrightarrow \text{, }S\in \s \text{ and }M'\in \s\cap\T \}$,
\item[(a3)] $\h_1=\B^+_1\cap \B^-_1$.
\item[(b1)] $\B^+_2=\{X\in \B \text{ }|\text{ } \mbox{there exists an}~\text{ } \EE\text{-triangle } V\to N\to X\dashrightarrow \text{, }V\in \V \text{ and }N\in \U\cap\V \}$,
\item[(b2)] $\B^-_2=\{Y\in \B \text{ }|\text{ } \mbox{there exists an}~ \text{ } \EE\text{-triangle } Y\to N'\to U\dashrightarrow \text{, }U\in \U \text{ and }N'\in \U\cap\V \}$,
\item[(b3)] $\h_2=\B^+_2\cap \B^-_2$.
\end{itemize}
By definition $\h_1/(\s\cap \T)$ is the heart of $(\s,\T)$, and $\h_2/(\U\cap \V)$ is the heart of $(\U,\V)$. For the convenience of the readers, we recall the following lemma.

\begin{lem}\label{FT}\cite[Lemma 4.2.]{LYZ}
Let $f:A\to B$ be a morphism in $\B$.
\begin{itemize}
\item[(1)] Assume that $A,B\in \h_1$, then $f$ factors through $\W$ if and only if it factors through $\s\cap \T$.
\item[(2)] Assume that $A,B\in \h_2$, then $f$ factors through $\W$ if and only if it factors through $\U\cap \V$.
\end{itemize}
\end{lem}

By this lemma, we have $\h_1/(\s\cap \T)=\underline {\h_1}$ and $\h_2/(\U\cap \V)=\underline {\h_2}$.

\medskip

By \cite[Proposition 4.1]{LYZ}, $\underline \h$ is abelian only if $\underline \h=\underline {\h_1}\cap \underline {\h_2}$. The following example shows that the conditions in Proposition \ref{prop2} are just necessary for the heart of $((\s,\T),(\U,\V))$ being abelian.

\begin{exm}\label{ex3}
Let $A=kQ/I$ be the algebra given in Example \ref{ex2}. We still denote by ``$\bullet$" in the quiver the objects belong to a subcategory and
by  ``$\circ $" the objects do not. Given two subcategories as follows:
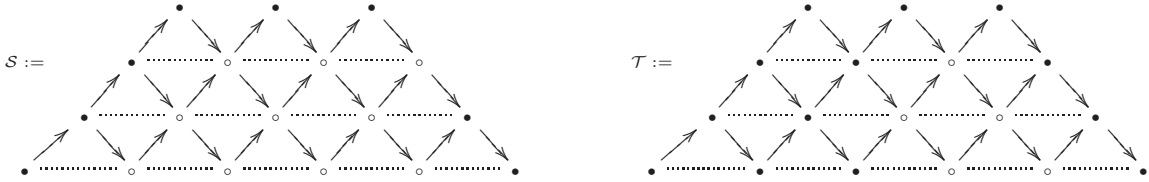
\begin{align}
	\tiny{\xymatrix @R=4mm @C2mm{&&&
			{\begin{smallmatrix}\bullet\end{smallmatrix}}\ar[rd]&&
			{\begin{smallmatrix}\bullet\end{smallmatrix}}\ar[rd]&&
			{\begin{smallmatrix}\bullet\end{smallmatrix}}\ar[rd]&&&\\
			\s:=&&{\begin{smallmatrix}\bullet\end{smallmatrix}}\ar[ru]\ar[rd]&&
			{\begin{smallmatrix}\circ\end{smallmatrix}}\ar[ru]\ar[rd]\ar@{.}[ll]&&
			{\begin{smallmatrix}\circ\end{smallmatrix}}\ar[ru]\ar[rd]\ar@{.}[ll]&&
			{\begin{smallmatrix}\circ\end{smallmatrix}}\ar[rd]\ar@{.}[ll]&&\\
			&{\begin{smallmatrix}\bullet\end{smallmatrix}}\ar[ru]\ar[rd]&&
			{\begin{smallmatrix}\circ\end{smallmatrix}}\ar[ru]\ar[rd]\ar@{.}[ll]&&
			{\begin{smallmatrix}\circ\end{smallmatrix}}\ar[ru]\ar[rd]\ar@{.}[ll]&&
			{\begin{smallmatrix}\circ\end{smallmatrix}}\ar[ru]\ar[rd]\ar@{.}[ll]&&
			{\begin{smallmatrix}\bullet\end{smallmatrix}}\ar[rd]\ar@{.}[ll]\\
			{\begin{smallmatrix}\bullet\end{smallmatrix}}\ar[ru]&&\textnormal{$\begin{smallmatrix}\circ\end{smallmatrix}$}\ar[ru]\ar@{.}[ll]&&
			{\begin{smallmatrix}\circ\end{smallmatrix}}\ar[ru]\ar@{.}[ll]&&\textnormal{$\begin{smallmatrix}\circ\end{smallmatrix}$}\ar[ru]\ar@{.}[ll]&&
			{\begin{smallmatrix}\circ\end{smallmatrix}}\ar[ru]\ar@{.}[ll]&&
			{\begin{smallmatrix}\bullet\end{smallmatrix}}\ar@{.}[ll]\\
	}}\notag
\ \ \quad\quad \quad\quad	\tiny{\xymatrix @R=4mm @C2mm{&&&
			{\begin{smallmatrix}\bullet\end{smallmatrix}}\ar[rd]&&
			{\begin{smallmatrix}\bullet\end{smallmatrix}}\ar[rd]&&
			{\begin{smallmatrix}\bullet\end{smallmatrix}}\ar[rd]&&&\\
			\T:=&&{\begin{smallmatrix}\bullet\end{smallmatrix}}\ar[ru]\ar[rd]&&
			{\begin{smallmatrix}\bullet\end{smallmatrix}}\ar[ru]\ar[rd]\ar@{.}[ll]&&
			{\begin{smallmatrix}\circ\end{smallmatrix}}\ar[ru]\ar[rd]\ar@{.}[ll]&&
			{\begin{smallmatrix}\bullet\end{smallmatrix}}\ar[rd]\ar@{.}[ll]&&\\
			&{\begin{smallmatrix}\bullet\end{smallmatrix}}\ar[ru]\ar[rd]&&
			{\begin{smallmatrix}\bullet\end{smallmatrix}}\ar[ru]\ar[rd]\ar@{.}[ll]&&
			{\begin{smallmatrix}\circ\end{smallmatrix}}\ar[ru]\ar[rd]\ar@{.}[ll]&&
			{\begin{smallmatrix}\circ\end{smallmatrix}}\ar[ru]\ar[rd]\ar@{.}[ll]&&
			{\begin{smallmatrix}\bullet\end{smallmatrix}}\ar[rd]\ar@{.}[ll]\\
			{\begin{smallmatrix}\bullet\end{smallmatrix}}\ar[ru]&&\textnormal{$\begin{smallmatrix}\bullet\end{smallmatrix}$}\ar[ru]\ar@{.}[ll]&&
			{\begin{smallmatrix}\bullet\end{smallmatrix}}\ar[ru]\ar@{.}[ll]&&\textnormal{$\begin{smallmatrix}\circ\end{smallmatrix}$}\ar[ru]\ar@{.}[ll]&&
			{\begin{smallmatrix}\circ\end{smallmatrix}}\ar[ru]\ar@{.}[ll]&&
			{\begin{smallmatrix}\bullet\end{smallmatrix}}\ar@{.}[ll]\\
	}}\notag
\end{align}
$(\s,\T)$ is a cotorsion pair and $\s\subset \T$. Let $\U=\T$ and
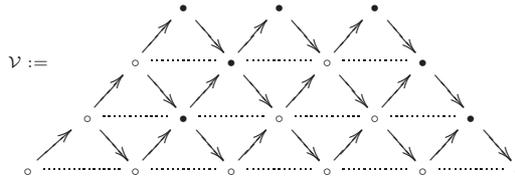
\begin{align}
	\tiny{\xymatrix @R=4mm @C2mm{&&&
			{\begin{smallmatrix}\bullet\end{smallmatrix}}\ar[rd]&&
			{\begin{smallmatrix}\bullet\end{smallmatrix}}\ar[rd]&&
			{\begin{smallmatrix}\bullet\end{smallmatrix}}\ar[rd]&&&\\
			\V:=&&{\begin{smallmatrix}\circ\end{smallmatrix}}\ar[ru]\ar[rd]&&
			{\begin{smallmatrix}\bullet\end{smallmatrix}}\ar[ru]\ar[rd]\ar@{.}[ll]&&
			{\begin{smallmatrix}\circ\end{smallmatrix}}\ar[ru]\ar[rd]\ar@{.}[ll]&&
			{\begin{smallmatrix}\bullet\end{smallmatrix}}\ar[rd]\ar@{.}[ll]&&\\
			&{\begin{smallmatrix}\circ\end{smallmatrix}}\ar[ru]\ar[rd]&&
			{\begin{smallmatrix}\bullet\end{smallmatrix}}\ar[ru]\ar[rd]\ar@{.}[ll]&&
			{\begin{smallmatrix}\circ\end{smallmatrix}}\ar[ru]\ar[rd]\ar@{.}[ll]&&
			{\begin{smallmatrix}\circ\end{smallmatrix}}\ar[ru]\ar[rd]\ar@{.}[ll]&&
			{\begin{smallmatrix}\bullet\end{smallmatrix}}\ar[rd]\ar@{.}[ll]\\
			{\begin{smallmatrix}\circ\end{smallmatrix}}\ar[ru]&&\textnormal{$\begin{smallmatrix}\circ\end{smallmatrix}$}\ar[ru]\ar@{.}[ll]&&
			{\begin{smallmatrix}\circ\end{smallmatrix}}\ar[ru]\ar@{.}[ll]&&\textnormal{$\begin{smallmatrix}\circ\end{smallmatrix}$}\ar[ru]\ar@{.}[ll]&&
			{\begin{smallmatrix}\circ\end{smallmatrix}}\ar[ru]\ar@{.}[ll]&&
			{\begin{smallmatrix}\bullet\end{smallmatrix}}\ar@{.}[ll]\\
	}}\notag
\end{align}
Then $(\U, \V)$ is a cotorsion pair and $((\s,\T),(\U,\V))$ is a twin cotorsion pair. In this case, $\W=\U=\T$. Obviously the conditions in Proposition \ref{prop2} are satisfied. We have
$$\underline \h=\uB=\add\bigg\{ \begin{smallmatrix}4\\3\end{smallmatrix}\oplus\begin{smallmatrix}5\\4\end{smallmatrix}\oplus\begin{smallmatrix}5\\4\\3\end{smallmatrix}\oplus\begin{smallmatrix}4\end{smallmatrix}\oplus\begin{smallmatrix}5\end{smallmatrix} \bigg\}.$$
$\underline \h$ is not abelian, since
$$\underline \h\nsubseteq \underline \h_1=\add\bigg\{ \begin{smallmatrix}4\\3\end{smallmatrix}\oplus\begin{smallmatrix}5\\4\\3\end{smallmatrix}\oplus\begin{smallmatrix}5\end{smallmatrix} \bigg\}$$
\end{exm}

We give the following sufficient conditions for the heart of $((\s,\T),(\U,\V))$ being abelian.

\begin{prop}\label{prop3}
{\rm (1)} If $\underline{\h}\subseteq\underline{\h_1}$ and ${_{\rm epi.}}\U\subseteq\s\oplus\W$,
then $\underline{\h}$ is abelian.
\vspace{2mm}

{\rm (2)} If $\underline{\h}\subseteq\underline{\h_2}$ and
$\T_{\rm mono.}\subseteq\V\oplus\W$, then $\underline{\h}$ is abelian.
\end{prop}

\begin{proof}
We only show (1), the proof of (2) is by duality.

To prove that $\underline{\h}$ is abelian,
we need to show that any monic-epic morphism in $\underline{\h}$ is an isomorphism.

Assume that $\underline{f}\colon A\to B$ is monic and epic in $\underline{\h}$. For convenience, we assume that
$A$ and $B$ have no direct summand in $\W$. Since $\underline \h\subseteq \underline {\h_1}$ and $\underline {\h_1}$ is abelian, we only need to check that $\underline f$ is also monic and epic in $\underline {\h_1}$
\vspace{1mm}

Since $B\in\h$, it admits an $\EE$-triangle $V_B \rightarrow W_B\xrightarrow{w_B} B \dashrightarrow$ with $W_B\in \W$ and $V_B\in \V$. Thus we can obtain a
commutative diagram
$$\xymatrix{
V_B \ar[r] \ar@{=}[d] &T \ar[r] \ar[d]&A \ar@{-->}[r] \ar[d]^f &\\
V_B \ar[r] &W_B \ar[r]^{w_B} &B \ar@{-->}[r] &
}
$$
which induces an $\EE$-triangle $T\longrightarrow A\oplus W_B\xrightarrow{(f~w_B)} B\dashrightarrow$, where $T\in \T$ since $\underline f$ is a monomorphism in $\underline \h$. Since $W_B$ admits an $\EE$-triangle $T_1\to M\xrightarrow{m} W_B\dashrightarrow$ with $M\in \s\cap \T$ and $T_1\in \T$, we have the following commutative diagram
$$\xymatrix{
T_1 \ar@{=}[r] \ar[d] &T_1 \ar[d]\\
T_2 \ar[r]\ar[d] &A\oplus M \ar[r]^-{(f~w_Bm)} \ar[d]^{\left( \begin{smallmatrix}1&0\\0&m\end{smallmatrix} \right)} &B \ar@{=}[d] \ar@{-->}[r] &\\
T \ar[r]\ar@{-->}[d] &A\oplus W_B \ar[r]_-{(f~w_B)} \ar@{-->}[d] &B \ar@{-->}[r] &\\
&&
}
$$
with $T_2\in \T$. The second row of the diagram above shows that $\underline{f}$ is monic in $\underline{\h_1}$.
\vspace{1mm}


Now assume that there is morphism $\underline{g}\colon B\to C$ in $\underline{\h_1}$ such that $\underline{gf}=0$ in $\underline {\h_1}$. We also assume that $C$ has no direct summand in $\W$. Then there exist two morphisms $m_1\colon A\to M'$
and $m_2\colon M'\to C$ such that $gf=m_2m_1$ where $M'\in\s\cap\T$.

Since $\underline f$ is an epimorphism in $\underline \h$, we can get an $\EE$-triangle
$$A\xrightarrow{\svecv{f}{w}} B\oplus W^A\longrightarrow U\dashrightarrow$$
where $W^A\in \W$, $U\in {_{\rm epi.}}\U$ and $w$ is $\W$-monic. So there exists a morphism $r\colon W^A\to M'$ such that
$m_1=rw$. It follows that
$$(g~-m_2r)\svecv{f}{w}=gf-m_2rw=gf-m_2m_1=0.$$
Since ${_{\rm epi.}}\U\subseteq\s\oplus\W$, let
$U=S\oplus W$ where $S\in\s$ and $W\in\W$.
We can rewrite the $\EE$-triangle above in the following way:
$$A\xrightarrow{\svecv{f}{w}} B\oplus W^A\xrightarrow{\left(\begin{smallmatrix}b& w_2\\
w_1&w_3\end{smallmatrix}\right)}S\oplus W\dashrightarrow.$$
Then there exists a morphism $(c_1~c_2)\colon S\oplus W\to C$
such that $(g~-m_2r)=(c_1~c_2)\left(\begin{smallmatrix}b& w_2\\
w_1&w_3\end{smallmatrix}\right)$.
In particular, we have $g=(c_1~c_2)\svecv{b}{w_1}=c_1b+c_2w_1$.

Since $C\in\h_1$, it admits an $\EE$-triangle $T_C\to M_C\xrightarrow{~u~}C\dashrightarrow$
where $T_C\in\T$ and $M_C\in\s\cap\T$.
Apply the functor $\Hom_{\B}(S,-)$ to the $\EE$-triangle above,
 we have the following exact sequence:
 $$\Hom_{\B}(S,M_C)\xrightarrow{\Hom_{\B}(S,\hspace{0.5mm}u)}
 \Hom_{\B}(S,C)\longrightarrow \EE(S,T_C)=0.$$
So there exists a morphism $m_1\colon S\to M_C$
such that $um_1=c_1$. This shows that
$c_1$ factors through $\s\cap\T$. Hence $g=c_1b+c_2w_1$ factors through $\W$. By Lemma \ref{FT}, $\underline{g}=0$ in $\underline{\h_1}$.
This shows that $\underline{f}\colon A\to B$ is epic in $\underline{\h_1}$. Thus $\underline f$ is an isomorphism.
\end{proof}

In summary, we have the following result.

\begin{thm}\label{main2}
$\underline \h$ is abelian if and only if the following conditions are satisfied:
\begin{itemize}
\item[(1)] $\underline \h=\underline {\h_1}\cap \underline {\h_2}$;
\vspace{1mm}

\item[(2)] ${_{\rm epi.}}\U\subseteq\s\oplus\W$;
\vspace{1mm}

\item[(3)] $\T_{\rm mono.}\subseteq\V\oplus\W$.
\end{itemize}
%
\end{thm}

\proof This follows from \cite[Proposition 4.1]{LYZ}, Proposition \ref{prop2} and Proposition \ref{prop3}.  \qed

\vspace{0.5cm}

\hspace{-4mm}\textbf{Data Availability}\hspace{2mm} Data sharing not applicable to this article as no datasets were generated or analysed during
the current study.
\vspace{2mm}

\hspace{-4mm}\textbf{Conflict of Interests}\hspace{2mm} The authors declare that they have no conflicts of interest to this work.

\end{document}